\documentclass[a4paper]{amsart}
\usepackage{amsmath,amssymb,amsfonts,amsthm,thmtools,mathtools,mleftright}
\usepackage{mathrsfs} 
\usepackage{nameref,hyperref}
\usepackage[nameinlink]{cleveref}
\usepackage{todonotes}
\setuptodonotes{backgroundcolor=yellow!40!white,bordercolor=white,noline,size=small}
\usepackage{tikz}
\usepackage{tikz-cd}
\usetikzlibrary{3d,calc}
\usepackage[shortlabels]{enumitem}
\usepackage[font={small,it}]{caption}

\newcommand\myshade{85}
\definecolor{mycitecolor}{rgb}{0.94, 0.8, 0.0}
\definecolor{mylinkcolor}{rgb}{0.56, 0.0, 1.0}
\definecolor{myurlcolor}{rgb}{0.5, 1.0, 0.83}
\hypersetup{
linkcolor  = mylinkcolor!\myshade!black,
citecolor  = mycitecolor!\myshade!black,
urlcolor   = myurlcolor!\myshade!black,
colorlinks = true,
}

\newcommand{\word}[1]{\textcolor{blue!40!black}{\emph{#1}}}

\declaretheorem[name=Theorem,refname={Theorem,Theorems},Refname={Theorem,Theorems},numberwithin=section,]{theorem}
\declaretheorem[name=Theorem,numbered=no,refname={Theorem,Theorems},Refname={Theorem,Theorems}]{mtheorem}

\declaretheorem[name=Lemma,refname={Lemma,Lemmas},Refname={Lemma,Lemmas},sibling=theorem,]{lemma}

\declaretheorem[name=Claim,numbered=no,]{claim*}

\declaretheorem[name=Fact,numbered=no]{fact*}

\declaretheorem[name=Definition,refname={Definition,Definitions},Refname={Definition,Definitions},sibling=theorem,style=definition,]{definition}

\declaretheorem[name=Remark,refname={Remark,Remarks},Refname={Remark,Remarks},sibling=theorem,style=remark,]{remark}
\declaretheorem[name=Remark,style=remark,numbered=no,]{remark*}

\newcommand{\Z}{\mathbb{Z}}
\newcommand{\N}{\mathbb{N}}
\newcommand{\R}{\mathbb{R}}
\newcommand{\Q}{\mathbb{Q}}
\newcommand{\V}[1]{\mathbb{V}(#1)}

\newcommand{\B}{\mathfrak{B}}
\newcommand{\seq}{\subseteq}
\newcommand{\les}{\preccurlyeq}
\newcommand{\df}{\coloneqq}
\DeclareMathOperator{\conv}{\tt conv}
\DeclareMathOperator{\vertices}{\tt vert}
\newcommand{\den}{\mathrm{den}}

\renewcommand{\leq}{\leqslant}
\renewcommand{\geq}{\geqslant}

\newcommand{\var}{\mathscr{V}}

\newcommand{\varprob}{{X}}
\newcommand{\varsol}{{Y}}
\newcommand{\varaux}{{Z}}

\newcommand{\eps}{\varepsilon}
\newcommand{\terms}[1]{T_{#1}(\LL)}

\DeclarePairedDelimiter\floor{\lfloor}{\rfloor}

\DeclareMathOperator{\dg}{\mathtt{d}}

\newcommand{\hooklongrightarrow}{\lhook\joinrel\longrightarrow}

\def\FF{\ensuremath{\mathcal F}}

\def\KK{\ensuremath{\mathcal K}}
\def\LL{\ensuremath{\mathcal L}}

\title[{\L}ukasiewicz unification with finitely many variables]{The unification type of {\L}ukasiewicz logic with a bounded number of variables}
\author{Marco~Abbadini}
\address{School of Computer Science, University of Birmingham, B15 2TT Birmingham, United Kingdom}
\email{m.abbadini@bham.ac.uk}
\author{Luca~Spada}
\address{Dipartimento di Matematica, Università degli Studi di Salerno, Piazza Renato Caccioppoli, 2, 84084 Fisciano (SA), Italy}
\email{lspada@unisa.it}



\begin{document}

\begin{abstract}
Building on the correspondence between finitely axiomatised theories in {\L}ukasiewicz logic and rational polyhedra, we prove that the unification type of the fragment of {\L}ukasiewicz logic with $n\geq 2$ variables is nullary. This solves a problem left open by V.\ Marra and L.\ Spada [Ann.\ Pure Appl.\ Logic 164 (2013), pp.\ 192–210]. 
Furthermore, we refine the study of unification with bounds on the number of variables. Our proposal distinguishes the number $m$ of variables allowed in the problem and the number $n$ in the solution. We prove that the unification type of {\L}ukasiewicz logic for all $m,n \geq 2$ is nullary. 
\end{abstract}

\maketitle

\section{Introduction} \label{s:introduction}

Unification is a fundamental concept in computer science, with applications in logic programming, automated reasoning, and type inference. At its core, a unification problem asks for a substitution that makes two different logical expressions identical. In automated reasoning, unification plays a crucial role in the resolution principle \cite{robinson1965machine}. In type systems, particularly within functional and logic programming languages, unification algorithms are employed to infer the most general types of expressions \cite{Hindley,MILNER1978348}.

When syntactical identity is replaced by equality modulo a given equational theory $E$, one speaks of $E$-unification. The study of unification modulo an equational theory has acquired increasing significance in the last few decades (see e.g. \cite{baader2001unification}).  The most basic property of $E$ in relation to unification issues is its \emph{unification type}. In rough terms, this determines the number of optimal solutions for a unification problem in the worst possible case. Unification modulo a theory finds applications in logic, where it is often used to study \emph{admissible rules} (see e.g., \cite{
10.1215/00294527-2009-004,
iemhoff2001admissible,
jerabek:admissible-rules-lukasiewicz,
jerabek:bases-lukasiewicz,
jerabek:complexity-lukasiewicz,
10.1093/logcom/exi029,
10.1093/jigpal/jzn004}).

In this paper, we are concerned with unification in {\L}ukasiewicz logic.
{\L}ukasiewicz logic is a non-classical many-valued logic that extends classical logic by allowing for more than two truth values. Named after the Polish logician Jan {\L}ukasiewicz, who introduced it in the early twentieth century (cf.\ \cite[Sec.~3]{LukTar1930} and \cite[pp.~38-59]{Tarski1956}), 
the infinitely-valued {\L}ukasiewicz logic generalises the concept of true and false to include infinitely many truth values, typically ranging from $0$ to $1$. This approach provides a flexible framework for modelling vagueness, gradations of truth, and uncertainty.

The unification type of {\L}ukasiewicz logic was studied in \cite{MSuni} where it was proved to be \emph{nullary}, i.e., the worst possible type. The proof of this result uses geometric methods and in particular involves the universal covering of the boundary of the square. The role of the universal covering is to provide more and more general solutions to a unification problem, thus showing that its solutions do not lie under any maximally general solution (formal definitions are given in \Cref{s:unification}).

The solutions provided by the universal covering involve an increasing number of variables. Thus, one naturally wonders whether bounding the number of variables allowed might lead to a better unification type. 

For example, this is the case for the equational theory of bounded distributive lattices.
Indeed, its unification type was proved to be nullary in \cite[Thm.\ 5.7]{Ghi1997}.
However, the variety is locally finite, and so, for every $n \in \N$, every unification problem has finitely many solutions with at most $n$ variables. It follows that, for every $n \in \N$, the unification type of bounded distributive lattices restricted to $n$ variables is either unitary or finitary.

The unification type of {\L}ukasiewicz logic with a bounded number of variables was conjectured to be nullary in \cite[p. 210]{MSuni}. In \cite{AbbaDiSpada}, a partial answer to the conjecture was provided: the authors proved that, for every $n \geq 2$, the unification type of the fragment of {\L}ukasiewicz logic with $n$ distinct variables is either infinitary or nullary.  The main contribution of this paper is the following.
\begin{mtheorem}[Main Result]\label{main:thm}
For all $m\geq 2$ and $n\geq 2$, the unification type of {\L}ukasiewicz logic restricted to at most $m$ variables for the problem and at most $n$ variables for the solutions is nullary. 
\end{mtheorem}

Regarding the remaining cases, note that if $m=0$ then the unification is trivially unitary since this is the case for every equational theory $E$. Furthermore, if $m \geq 1$ and $n = 0$ then the unification is finitary (e.g.\ the problem $x \approx x$ has precisely two solutions: $0$ and $1$; and they are incomparable).
The case $m=1$ and $n\geq 1$ gives a finitary unification type (see \cite[Section 4]{MSuni}).
It remains an open problem to determine the unification type restricted to at most $n=1$ variable for the solution and at most $m \geq 2$ variables for the problem. (We are not able to address this case because \cref{l:make-space} below fails for $n=1$.)

The following table summarises the unification types of {\L}ukasiewicz logic restricted to at most $m\in\{0,1,2,\dots\}\cup\{\aleph_0\}$ variables for the problem and at most $n\in\{0,1,2,\dots\}\cup\{\aleph_0\}$ for the solution.
\begin{center}
\begin{tabular}{| r | c c c |}
\hline
& $n=0$ & $n=1$ & $n\geq 2$ \\
\hline
$m=0$ & unitary & unitary & unitary\\
$m=1$ & finitary & finitary& finitary\\
$m\geq 2$ & finitary & ? & nullary\\
\hline  
\end{tabular}
\end{center}\vspace{1em}

The proof of the \nameref{main:thm} is performed by exhibiting a specific unification problem that witnesses the nullarity of the unification type. More precisely, we show:
\begin{fact*}\label{main:lemma}
The unification problem 
\begin{align}\label{eq:unification-problem}
x_1 \lor x_2 \lor \lnot x_1 \lor \lnot x_2 \approx 1
\end{align}
in the variables $x_1$ and $x_2$ has a unifier $\iota$ in the variable $y$, namely
\[
\iota(x_1) = y \text{ and }\iota(x_2) = 0,
\]
such that, for every $n \geq 2$ and unifier $\sigma$ in at most $n$ variables more general than $\iota$, there is a unifier $\sigma'$ in at most $n$ variables strictly more general than $\sigma$.
\end{fact*}

Notice that the unification problem \eqref{eq:unification-problem} is the same that was used in \cite{MSuni} to prove the nullarity of the unification type of the full {\L}ukasiewicz logic. However, the proof strategy in this case is completely different.  We summarise here the main steps.  

First, in \Cref{s:unification} we refine Ghilardi's algebraic approach in order to account for bounds on the number of variables in the problem and in the solutions (\Cref{l:iso}); this allows us to recast the unification problem with a bounded number of variables in terms of the algebraic semantics of {\L}ukasiewicz logic: MV-algebras.  As a second step, we use the duality between finitely presented MV-algebras and rational polyhedra (see \Cref{s:lukasiewicz}) to transform the algebraic problem into a geometric one.  The basic concepts and results in polyhedral geometry needed in the rest of the paper are recalled in this section.

Under the above duality, the unification problem \eqref{eq:unification-problem} corresponds to the boundary $\B$ of the unit square $[0,1]^2$. Unifiers correspond in the duality to piecewise affine maps with integer coefficients.  We prove (i) that every unifier $[0,1]^n \to \B$ more general than
\begin{align*}
\iota' \colon [0,1] &\hooklongrightarrow \B\\
x &\longmapsto (x,0)
\end{align*}
is not constant on the boundary of $[0,1]^n$ (\cref{l:not-constant-is-preserved}) and (ii) that every unifier $[0,1]^n \to \B$ that is not constant on the boundary of $[0,1]^n$ admits a strictly more general unifier $[0,1]^{\max\{2, n\}} \to \B$ (this is \cref{p:make-space}, if we take into account the first observation in the proof of \cref{l:technical}).
Together, these two claims prove the Fact above and hence the nullarity, for any $m,n \geq 2$, of the unification type of {\L}ukasiewicz logic restricted to at most $m$ variables for the problem and at most $n$ for the solutions. Although claim (i) is quite immediate, claim (ii) will require more technical work, which we do in \Cref{s:squeeze,s:strictly-more-general}.

\section{Unification}
\label{s:unification}

Let $\var$ be a fixed infinite set of variables. We fix a purely functional language $\LL=(F,\alpha)$ in the usual sense, i.e.\ $F$ is a set of function symbols and $\alpha$ an arity function. For $\varprob\seq\var$ we write $\terms{\varprob}$ for the set of all terms in $\LL$ with variables ranging in $\varprob$. 

As usual, by an \word{equation} we mean a pair $(t_1,t_2)$ of terms in $\terms{\varprob}$ (usually written as $t_1\approx t_2$).  Let $E$ be a set of equations. By $t_1\approx_{E}t_2$ we mean that the universal closure of the equation $t_1 \approx t_2$ holds in every $\LL$-algebra that satisfies the universal closure of every equation from $E$. 
For $\varprob,\varsol\seq\var$, a map $\sigma\colon \varprob\to \terms{\varsol}$ is called a \word{substitution}. Every substitution $\sigma\colon \varprob\to \terms{\varsol}$ extends in a unique way to a map $\widehat{\sigma}\colon \terms{\varprob}\to \terms{\varsol}$ that commutes with the function symbols in $\LL$. 

An \word{$E$-unification problem} is a pair $(\varprob, S)$ where $X$ is a finite subset of $\var$ and $S=\{s_1\approx t_1,\dots, s_k\approx t_k\}$ is a finite set of equations with variables in $X$.
An \word{$E$-unifier for $(X,S)$ with variables in $\varsol$} is a substitution $\sigma\colon \varprob\to \terms{\varsol}$ such that
\begin{equation*}
        \widehat{\sigma}(s_1)\approx_E\widehat{\sigma}(t_1),\quad \dots,\quad \widehat{\sigma}(s_k)\approx_{E}\widehat{\sigma}(t_k). 
\end{equation*}
For $n\in\{0,1,2,\dots,\}\cup\{\aleph_0\}$, we denote by $U_E(n,X,S)$ the set of all $E$-unifiers for $(X,S)$ with variables\footnote{Notice that, although we require the unifiers in $U_E(n,X,S)$ to have their codomain equal to $\terms{\varsol}$, we are not requiring them to mention all variables in $\varsol$.} in $Y$, with $Y$ ranging among finite subsets of $\var$ of cardinality at most $n$.

 Let $\varsol_{1},\varsol_{2}\seq\var$ and let $\sigma\colon \varprob\to \terms{\varsol_1}$ and $\tau\colon \varprob\to \terms{\varsol_2}$ be substitutions. We say that $\sigma$ is \word{more general} than $\tau$ (with respect to $E$), in symbols $\tau\les_{E}\sigma$, if there is a substitution $\theta\colon \varsol_1\to \terms{\varsol_2}$ such that $\tau(x)\approx_{E}\widehat{\theta}(\sigma(x))$ is valid for every $x\in X$. The relation $\les_{E}$ is a preorder, i.e., a reflexive and transitive relation.
We consider the restriction of $\les_{E}$ to $U_E(n,X,S)$. It is well known that any preordered set $(P,\les)$ induces a partial order ---called the \word{poset reflection}--- on the quotient $P/{\sim}$, where $p\sim q$ if and only if $p\les q$ and $q\les p$,  for any $p,q\in P$. With an abuse of notation, we denote in the same way the pre-ordered set and its partially ordered quotient.

We say that a subset $S$ of a partially ordered set $(P,\leq)$ \word{covers} $P$ if every $p\in P$ is less than or equal to some $s\in S$.
We say that a non-empty partially ordered set $(P,\leq)$ has:
\begin{enumerate}
\item  \word{unitary type} if the set of maximal elements of $(P,\leq)$ covers $P$ and has cardinality $1$;
\item \word{finitary type} if the set of maximal elements of $(P,\leq)$ covers $P$ and has finite cardinality strictly greater than $1$;
\item \word{infinitary type} if the set of maximal elements of $(P,\leq)$ covers $P$ and is infinite;
\item \word{nullary type} if the set of maximal elements of $(P,\leq)$ does not cover $P$.
\end{enumerate}
Exactly one of the above conditions holds. It is understood that the list above is arranged in decreasing order of desirability.  The \word{unification type of $E$} is defined as the least desirable type occurring among all non-empty $U_{E}(\aleph_0,\varprob,S)$, for $(\varprob, S)$ that ranges among all $E$-unification problems.
More generally, given $m,n\in \{0,1,2,\dots\}\cup\{\aleph_0\}$, the \word{unification type of $E$ restricted to at most $m$ variables for the problem and at most $n$ variables for the solution} is defined as the least desirable type occurring among all non-empty $U_{E}(n,X,S)$, for $(\varprob, S)$ that ranges among all $E$-unification problems with $\varprob$ of cardinality at most $m$.
\begin{remark}
If in the above definition one or both of the ``at most''  are replaced with ``precisely'' one gets an equivalent definition. This is particularly easy to see if the language has at least a constant symbol $c$. In this case, for any $m' \leq m$, every problem in variables $x_1, \dots, x_{m'}$ can be turned into a problem in the variables $x_1, \dots, x_m$ by adding for any $i \in \{m' + 1, \dots, m\}$ the equation $x_i \approx c$. This shows that the ``at most'' regarding the problem can be turned into a ``precisely''. Moreover, for any $n' \leq n$, every unifier in variables $y_1, \dots, y_{n'}$ is equally general to itself seen in variables $y_1, \dots, y_n$, as witnessed by the substitution $\{y_1, \dots, y_{n'}\} \hookrightarrow T_{\{y_1, \dots, y_n\}}$ mapping $y_i$ to $y_i$ and the substitution $\{y_1, \dots, y_{n'}\} \to T_{\{y_1, \dots, y_n\}}$ mapping $y_i$ to $y_i$ if $i \leq n'$ and to $c$ otherwise. This shows that also the ``at most'' regarding the solution can be turned into a ``precisely''.

The same is true also without the hypothesis that the language has at least a constant symbol. First of all, we notice that for all possible choices for the number of variables in the solutions, a unifiable problem in $0$ variables will have a unitary type, and so problems in $0$ variables can be disregarded. For any $1 \leq m' \leq m$, every problem in variables $x_1, \dots, x_{m'}$ can be thought of as a problem in the variables $x_1, \dots, x_{m'}, \dots, x_m$ by adding for any $i \in \{m' + 1, \dots, m\}$ the equation $x_i \approx x_1$. Finally,  observe that if the problem is in at least one variable, say $x$, then for all $n' \leq n$, every unifier $\sigma$ in variables $\{y_1, \dots, y_{n'}\}$ is equally general to itself seen in variables $\{y_1, \dots, y_{n}\}$, as witnessed by the substitution $\{y_1, \dots, y_{n'}\} \hookrightarrow T_{\{y_1, \dots, y_n\}}$ mapping $y_i$ to $y_i$ and the substitution $\{y_1, \dots, y_{n'}\} \to T_{\{y_1, \dots, y_n\}}$ mapping $y_i$ to $y_i$ if $i \leq n'$ and to $\sigma(x)$ otherwise.
\end{remark}

In \cite{Ghi1997}, Ghilardi introduced an algebraic approach to unification modulo an equational theory $E$ through the notions of finitely presentable and projective objects. Since these notions are categorical, the $E$-unification types can be studied up to a categorical equivalence without knowing how the equivalence functors are defined.

\begin{remark}
There cannot be any categorical characterisation of algebras that are ``presentable with $n$ variables'', as we briefly show here. 
Let $\mathbb{T}$ be the variety generated by any primal algebra with three elements. By \cite[Thm.\ 1]{Hu1969} (see also \cite[Cor., p.\ 153]{Hu1971}), $\mathbb{T}$ is dually equivalent to the category of Stone spaces, and thus it is also equivalent to the category of Boolean algebras.
The free Boolean algebra on $1$ generator corresponds under Stone duality to the $2$-element space.
The free $\mathbb{T}$-algebra on $1$ generator corresponds under this duality to the $3$-element space.
Therefore, under these dualities, the 1-generated Boolean algebras correspond to the Stone spaces with at most two elements, and the 1-generated $\mathbb{T}$-algebras to the Stone spaces with at most three elements.
Thus, the $\mathbb{T}$-algebra freely generated on $1$ generator corresponds to a Boolean algebra (namely, the power set of a set of three elements) that is not $1$-generated.
We conclude that the notion of ``$1$-generated'' is not preserved under categorical equivalence.
\end{remark}

Therefore, we work with free algebras instead of projective ones.
However, we will still solve the problem through a categorical duality; this will be possible because we will keep track of the objects that in our specific duality correspond to algebras that are presentable with $n$ variables.

Given a set of equations $E$ in a purely functional language $\LL$, let $\V{E}$ be the class of its models; $\V{E}$ is said to be a \word{variety}.
A \word{presentation} is a pair $(\varprob,S)$ consisting of a set $\varprob\seq \var$ and a set of equations $S$ with variables in $\varprob$. 
To each presentation $(\varprob,S)$ we associate the quotient of the $\V{E}$-algebra freely generated by $X$ by the congruence on this algebra generated by the image of $S$ under the canonical homomorphism from terms to elements of the free algebra.
We write $\FF(\varprob)$ as a shorthand for $\FF(\varprob,\emptyset)$; the algebra $\FF(\varprob)$ is known as the $\V{E}$-algebra freely generated by $\varprob$. An algebra $A$ is called \word{finitely presentable} if there is a presentation $(\varprob,S)$ with $\varprob$ and $S$ finite such that $A$ is isomorphic to $\FF(\varprob,S)$.  

Let $f\colon\FF(\varprob,S)\to \FF(\varaux_{1})$ and $g \colon \FF(\varprob,S)\to\FF(\varaux_{2})$ be $\LL$-homomorphisms with $Z_1,Z_2\seq \var$. We set $f\les g$ if there is an $\LL$-homomorphism $h\colon  \FF(\varaux_{2})\to \FF(\varaux_{1})$ such that $f=h\circ g$.
Let $V_E(n, X, S)$ be the set of $\LL$-homomorphisms from $\FF(X,S)$ to $\FF(\varaux)$, with $\varaux$ ranging among subsets of $\var$ of cardinality $n$.  
\begin{lemma} \label{l:iso}
    For every $E$-unification problem $(X,S)$ and every $n\in \{0,1,2,\dots\}\cup\{\aleph_0\}$, the partially ordered sets $U_E(n,X,S)$ and $V_E(n,X,S)$ are isomorphic.
\end{lemma}
\begin{proof}
    The ensuing argument is just a minor variation of \cite[Theorem 4.1]{Ghi1997}.
    For the sake of the reader, we explicitly provide the bijection.

    To any $E$-unifier $\sigma$ for $S$ with variables in $\varsol$, we associate the unique homomorphism $f_{\sigma}\colon \FF(\varprob,S)\to \FF(\varsol)$ that makes the following diagram commute:
    \[
        \begin{tikzcd}
            X \arrow{r}{} \arrow{d}{\sigma} & \FF(X) \arrow[two heads]{r}{} & \FF(X,S) \arrow[dashed]{d}{f_{\sigma}}\\
            T_Y \arrow[two heads]{rr}{} && \FF(Y),
        \end{tikzcd}
    \]
    i.e., the homomorphism that sends the equivalence class of a term $t$ to the equivalence class of $\widehat{\sigma}(t)$.  
    It is routine to check that $f_{\sigma}$ is a well-defined homomorphism precisely because $\sigma$ is a unifier for $S$. 
    
    Vice versa, let $f\colon \FF(\varprob,S)\to \FF(\varsol)$ be a $\V{E}$-homomorphism.
    For each $x \in X$, we let $[x]$ denote the equivalence class of $x$ in $\FF(\varprob,S)$, and we choose an element $\sigma_f(x) \in \terms{\varsol}$ such that its equivalence class in $\FF(\varsol)$ is $f([x])$.
    The resulting function $\sigma_{f} \colon X \to \terms{\varsol}$ is a $E$-unifier for $S$ with variables in $S$: indeed, if $(s(x_{1},\dots,x_{n}),t(x_{1},\dots, x_{n}))\in S$, then
    \begin{align*}
        \widehat{\sigma_{f}}(s(x_{1},\dots,x_{n}))&\approx_{E}s(f([x_{1}]),\dots,f([x_{n}]))\approx_{E}f\left([s(x_{1},\dots,x_{n})]\right)\\
        &\approx_{E}f\left([t(x_{1},\dots,x_{n})]\right) \approx_{E}t(f([x_{1}]),\dots,f([x_{n}]))\\
        & \approx_{E} \widehat{\sigma_{f}}(t(x_{1},\dots,x_{n})).
    \end{align*}
    We then have the following commutative diagram.
    
    \[
        \begin{tikzcd}
            X \arrow{r}{} \arrow[dashed]{d}{\sigma_f} & \FF(X) \arrow[two heads]{r}{} & \FF(X,S) \arrow{d}{f}\\
            T_Y \arrow[two heads]{rr}{} && \FF(Y)
        \end{tikzcd}
    \]
    
    It is not difficult to see that the composite $V_E(n,X,S) \to U_E(n,X,S) \to V_E(n,X,S)$ is the identity.
    Moreover, it is equally easy to see that the composite $U_E(n,X,S) \to V_E(n,X,S) \to U_E(n,X,S)$ maps a $E$-unifier $\sigma$ to an $E$-unifier $\sigma'$ such that for every $x \in X$ we have $\sigma(x) \approx_E \sigma'(x)$, and so $\sigma \les_E \sigma'$ and $\sigma' \les_E \sigma$.
    
    Both the above-defined maps are order-preserving.
    Indeed, let $\tau$ and $\sigma$ be $E$-unifiers for $S$ with variables in $\varsol_1$ and $\varsol_2$, respectively, and suppose that $\tau \les_E \sigma$.
    Then there is a substitution $\theta\colon \varsol_1\to \terms{\varsol_2}$ such that $\tau(x)\approx_{E}\widehat{\theta}(\sigma(x))$ holds for every $x\in X$.
    This function induces a homomorphism $\FF(\theta) \colon \FF(\varsol_1) \to \FF(\varsol_2)$ (obtained by applying the free functor $\FF$ to $\theta$) which satisfies $\FF(\theta) \circ f_\sigma = f_\tau$, from which we deduce $f_\tau \les f_\sigma$.
    Then, the restriction of $h$ to $\varaux_1$ gives a substitution $h' \colon \varaux_{1} \to \FF(\varaux_2)$ such that $\sigma_f = \widehat{h'} \circ \sigma_g$, whence $\sigma_f \les \sigma_g$.
\end{proof}

\section{{\L}ukasiewicz logic, MV-algebras and rational polyhedra}
\label{s:lukasiewicz}

The equivalent algebraic semantics of {\L}ukasiewicz logic is provided by MV-algebras.  The standard reference is \cite{CignoliDOttavianoEtAl2000}. An important tool in the study of MV-algebras is the duality between finitely presented MV-algebras and rational polyhedra (see \cite{marra_spada,MSuni} and also \cite{cabrer2017mv}), which we will introduce after briefly recalling some concepts in piecewise geometry needed for the presentation. Since we will only work in the dual category of rational polyhedra, the algebraic properties of MV-algebras have no relevance in this paper.

For any $S\seq \R^{n}$, the \word{convex hull} of $S$ is defined as 
\[
\conv(S)\df \left\{\sum_{i=1}^{k}\lambda_{i}s_{i}\in \R^{n}\mid  k\in \N, s_{i}\in S,\, \lambda_{i}\geq 0 \text{ such that } \sum_{i = 1}^k \lambda_{i} =1 \right\}.
\]
A \word{rational polytope} in $\R^n$ is the convex hull of finitely many points with rational coordinates.  A \word{rational polyhedron} is a finite union of rational polytopes.  
\begin{theorem}[{\cite[Theorem 1.1]{Ziegler1995}}]\label{t:hull-intersection}
    A subset $P \subseteq \R^n$ is the convex hull of finitely many points if and only if it is a bounded intersection of finitely many half-spaces.
\end{theorem}

Recall that a map $\eta \colon \R^n\to \R$ is said to be \word{affine} if there are $\lambda_0,\lambda_1,\dots,\lambda_n\in \R$ (the \word{coefficients} of $\eta$) such that $\eta(x_1,\dots,x_n)=\lambda_0+\lambda_1x_1+\dots+\lambda_nx_n$ for every $(x_1,\dots,x_n)\in \R^n$. 
\begin{definition}
Let $m,n\in\N$. A function $\eta \colon \R^m\to \R^n$ is called a \word{$\Z$-map} if it is continuous and there are affine maps $\eta_1, \dots, \eta_k$ with integer coefficients such that for any $x\in \R^m$ there is $i\leq k$ for which $\eta(x) = \eta_i(x)$. For $A\seq \R^m$ and $B\seq \R^n$, we also call $\Z$-map any function $\eta'\colon A \to B$, obtained by restricting a $\Z$-map $\eta\colon \R^m \to \R^n$, i.e.\ such that for any $x\in A$, $\eta'(x)=\eta(x)\in B$.
\end{definition}

\begin{theorem}[{\cite[Theorem~3.4]{MSuni}}]\label{d:duality}
    The category of rational polyhedra and $\Z$-maps is dually equivalent to the category of finitely presented MV-algebras and homomorphisms.
\end{theorem}

Under the duality of \cref{d:duality}, for each $n\in\N$ the free MV-algebra on $n$ generators corresponds to the Thyconoff cube $[0,1]^n$.
The dual of a unification problem in $m$ variables $S$ is a rational polyhedron $B_S\seq [0,1]^{m}$. The dual of a unifier with $n$ variables for $S$ is a $\Z$-map $\eta\colon [0,1]^{n}\to B_S$. 
Finally, if the $\Z$-maps $\tau\colon [0,1]^{n_1}\to B$ and $\sigma\colon [0,1]^{n_2}\to B$ correspond to some unifiers $\tau'$ and $\sigma'$, respectively, then $\sigma'$ is more general than $\tau'$ if and only if there is a $\Z$-map $\alpha\colon [0,1]^{n_1} \to [0,1]^{n_2}$ such that the following diagram commutes.
\begin{equation*}
\begin{tikzcd}
{[0,1]^{n_2}}\arrow{r}{\sigma}&B\\
{[0,1]^{n_1}}\arrow[swap]{ru}{\tau}\arrow[dashed]{u}{\alpha}&
\end{tikzcd}
\end{equation*}
Thus, in the following we will often say that a $\Z$-map $\sigma$ is more general than a $\Z$-map $\tau$ when the condition above applies. Finally, the dual of the unification problem $x_1 \lor x_2 \lor \lnot x_1 \lor \lnot x_2 \approx 1$ introduced in \Cref{s:introduction} is the boundary $\B$ of the unit square $[0,1]^2$, i.e.,
\[
\conv\{(0,0),(1,0)\}\cup\conv\{(1,0),(1,1)\}\cup\conv\{(1,1),(0,1)\}\cup\conv\{(0,1),(0,0)\}.
\]

Having translated unification problems of {\L}ukasiewicz logic in properties of rational polyhedra and $\Z$-maps, we now recall some standard tools in piece-wise linear geometry.

A set $W\seq \R^n$ is an \word{affine space} if either $W=\emptyset$ or there are $x \in \R^n$ and a vector subspace $V\seq \R^n$ such that $W = x + V$.
In the latter case, we define the \word{direction} of $W$ as the vector space $V$ and the dimension of $W$ as the dimension of $V$, denoted by $\dim{W}$.
Conventionally, $\dim\emptyset\df-1$. 
We also say that any element $v \in V$ is \word{parallel} to the affine space $x + V$.
For $S \subseteq \R^n$, the \word{affine span} of $S$ is the intersection of all the affine spaces in $\R^n$ that contain $S$. The vectors $v_0,\dots,v_k\in \R^n$ are called \word{affinely independent} if the affine span of $\{v_0,\dots,v_k\}$ has dimension $k$.

Let $k\in \{-1,0,1,2,\dots\}$. A \word{rational simplex of dimension $k$} (also called \word{rational $k$-simplex}) $S$ in $\R^n$ is the convex hull of $k+1$ affinely independent $v_0,\dots,v_{k}\in \Q^n$.  Given a rational $k$-simplex $S$, there is a unique ($k+1$)-tuple of points $v_0,\dots,v_{k}$ such that
\[\conv\left(\{v_0,\dots,v_k\}\right)=S.\]
These points are called the \word{vertices} of $S$ and denoted by $\vertices(S)$.

A \word{face} of a simplex $S$ is the convex hull of any subset of $\vertices(S)$.  
A \word{rational} \word{simplicial complex} $\KK$ in $\R^n$ is a finite set of rational simplices in $\R^n$ that satisfies the following conditions:
\begin{enumerate}
\item every face of a simplex in $\KK$ is also in $\KK$;
\item if $S_{1},S_2\in \KK$, then $S_1\cap S_2$ is a face of both $S_1$ and $S_2$.
\end{enumerate}

For every finite set $\mathcal{K}$ of rational simplices in $\R^n$, the union of the simplices of $\mathcal{K}$ is denoted by $\lvert \mathcal{K}\rvert$. 
Every rational simplicial complex $\KK$ is said to be a \word{rational triangulation} of $\lvert\KK\rvert$. 
The \word{vertices} of a rational simplicial complex are exactly the vertices of its simplices; we denote by $\vertices(\KK)$ the set of vertices of $\KK$.

For any $v=(v_1,\dots,v_n)\in \Q^n$, the \word{denominator} of $v$, denoted by $\den(v)$, is the least common multiple of the denominators of $v_1$, \dots, $v_n$. Moreover, we set
\[
\widetilde{v}\df\den(v)(v,1)=(\den(v)v_1,\dots,\den(v)v_n, \den(v))\in \Z^{n+1}.
\]
A simplex $\conv(w_0,\dots,w_k)\seq \R^n$ is said to be \word{regular} if its vertices are rational and the set of integer vectors $\{\widetilde{w_0}, \dots,\widetilde{w_k} \}$ can be extended to a basis of the free abelian group $\Z^{k+1}$. A \word{regular triangulation} is a rational triangulation whose simplices are regular.

\begin{theorem}[{\cite[Thm.~2.2]{RourkeSanderson1972}}]\label{l:triangulation}
For every $\Z$-map $\eta \colon P \to Q$ there is a regular triangulation $\mathcal{K}$ of $P$ such that the restriction of $\eta$ to each simplex of $\mathcal{K}$ is an affine map with integer coefficients.
\end{theorem}
\begin{lemma}\label{l:strong-triangulation}
For every finite set $\KK$ of rational polyhedra in $\R^n$, there is a regular triangulation $\Delta$ of $\lvert \KK\rvert$ such that every element of $\KK$ is a union of simplices of $\Delta$.
\end{lemma}
\begin{proof}
The proof of \cite[Proposition 1]{Mun2008} gives the result for the case in which every element of $\KK$ is contained in $[0,1]^n$. However, the slightly more general result in the statement can be obtained in an analogous way by covering $\lvert \KK\rvert$ with translations of the unit cube by vectors with integer coordinates.
\end{proof}

\begin{lemma}\label{l:unique-extension}
Let $P\seq\R^n$ be a rational polyhedron, $\Delta$ a regular triangulation of $P$ and $f\colon \vertices(\Delta)\to \Q^m$ a function. If $\den(f(v))$ divides $\den(v)$ for every $v\in \vertices(\Delta)$, then $f$ can be uniquely extended to a $\Z$-map $\eta\colon P\to \R^m$ that is affine on each simplex of $\Delta$.
\end{lemma}
\begin{proof}
This was established in \cite[Lemma 5.1]{MunCab2012} for polyhedra in $[0,1]^n$, but a similar argument also works for polyhedra in $\R^n$. 
\end{proof}

Finally, we recall that the \word{relative interior} of a set $S$ is defined as the interior of $S$ within its affine hull.
Given any presentation of a $k$-simplex $S$ in $\R^n$ as the solution set of a system of $n-k$ affine equations and $k + 1$ affine inequalities, the relative interior of $S$ is the solution set of the same equations and the strict versions of the affine inequalities.

\section{Squeezing lemma}
\label{s:squeeze}

We are now going to prove a series of lemmas on polyhedra and $\Z$-maps.  Our final aim is to show that, for every $n\geq 2$, any $\Z$-map $[0,1]^n \to \B$ that is not constant on the boundary of $[0,1]^n$ can be slightly modified to obtain a strictly more general $\Z$-map $[0,1]^n \to \B$ that is still not constant on the boundary of $[0,1]^n$; this is the content of \cref{t:exists-upper-bound} below. 
\begin{lemma}\label{l:exists-rational}
Let $Q$ be a rational polyhedron in $\R^n$ and $P$ a closed subset of $\R^n$. If $Q\setminus P\neq \emptyset$ then there is a rational element in $Q\setminus P$.
\end{lemma}

\begin{proof}
Let $x\in Q\setminus P$, which is non-empty by hypothesis. By \cref{l:strong-triangulation} there are rational simplices $S_1,\dots,S_n$ such that $Q=S_1\cup\dots\cup S_n$. Thus, there is $i\in \{1,\dots,n\}$  such that $x\in S_i\setminus P$. Let $\vertices{S_{i}}=\{v_1,\dots,v_m\}$ and let
\[\Delta_m\df\big\{(a_1,\dots,a_m)\in [0,1]^m\mid a_1+\dots+a_m=1 \big\}.\]
It follows that there is $(\lambda_1,\dots,\lambda_m)\in \Delta_m$ such that $x=\lambda_1 v_1+\dots +\lambda_m v_m$.
Using the density of rational numbers, for every $\varepsilon>0$ one can find $\delta_2, \dots, \delta_m\in(-\varepsilon,\varepsilon)$ such that 
\begin{align}\label{eq:delta-epsilon}
\lambda_1-1\leq\delta_2+\dots+\delta_m\leq \lambda_1\text{ and }\lambda_2+\delta_2,\dots,\lambda_m+\delta_m\in \Q\cap[0,1].
\end{align}
Notice that $\lambda_1-(\delta_2+\dots+\delta_m)$ must also be a rational number, because 
\begin{align*}
\lambda_1-(\delta_2+\dots+\delta_m)&=1-(\lambda_2+\dots+\lambda_m)-(\delta_2+\dots+\delta_m)\\
&=1-((\lambda_2+\delta_2)+\dots+(\lambda_m+\delta_m))
\end{align*}
and the latter is a sum of rational numbers.
Moreover, by \cref{eq:delta-epsilon},  $0\leq \lambda_1-(\delta_2+\dots+\delta_m)\leq 1$.
Therefore, there are rational points in $\Delta_m$ arbitrarily close to $(\lambda_1,\dots,\lambda_m)$.
Since the map from $\Delta_m$ to $Q$ defined by $(\alpha_1, \dots, \alpha_m) \mapsto \alpha_1 v_1 + \dots + \alpha_m v_m$ is continuous and maps rational points to rational points, there are rational points in $Q$ arbitrarily close to $\lambda_1 v_1 + \dots + \lambda_m v_m = x$.
Since the complement of $P$ is open and contains $x$, it must also contain one of these points.
\end{proof}

\begin{lemma}\label{l:extend-super}
Let $P$ and $Q$ be rational polyhedra in $\R^n$ with $P$ strictly included in $Q$. For any $z \in \Z$, 
every $\Z$-map from $P$ to $\R$ can be extended to a $\Z$-map from $Q$ to $\R$ whose image contains $z$.
\end{lemma}
\begin{proof}
Let $\eta \colon P \to \R$ be a $\Z$-map.
By \cref{l:triangulation} there is a triangulation $\KK$ of $P$ such that $\eta$ is affine over each simplex of $\KK$. Let $s$ be a rational element of $Q \setminus P$, whose existence is guaranteed by \cref{l:exists-rational}.
By an application of \cref{l:strong-triangulation} to $\KK\cup\{Q,\{s\}\}$, there is a regular triangulation $\Delta$ of $Q$ such that $s\in \vertices(\Delta)$ and every simplex of $\KK$ is a union of simplices of $\Delta$. We define
\begin{align*}
    f \colon \vertices(\Delta) & \longrightarrow \R\\
    v& \longmapsto {\begin{cases}
    \eta(v)&\text{if }v\in P,\\
    z&\text{otherwise.}
    \end{cases}}
\end{align*}
Since $\eta$ is a $\Z$-map and $z$ has denominator $1$, $f(v)$ is rational and $\den(f(v))$ divides $\den(v)$ for every $v \in \vertices(\Delta)$. Thus, by \cref{l:unique-extension}, $f$ can be uniquely extended to a $\Z$-map $\theta\colon Q\to \R$ that is affine on each simplex of $\Delta$. Since $\eta$ is affine over every simplex of $\KK$ and every simplex of $\KK$ is a union of simplices of $\Delta$, the map $\theta$ extends $\eta$. Moreover, $\theta(s) = f(s) = z$.
\end{proof}

\begin{lemma}\label{l:hyperplane}
Let $t,v,w\in\Q^n$, with $v\neq 0$. For every $\eps>0$, there are rational numbers $\delta, \tau\in(0,\eps)$ such that $\den(t+\delta v)=\den(t+\delta v +\tau w)$.
\end{lemma}
\begin{proof}
It is easy to see that there are strictly positive integers $a$ and $b$ such that, setting $v' \df av$ and $w' \df bw$, the vectors $v'$ and $w'$ belong to $\Z^n$ and $v' + w' \neq 0$.
We choose a prime $p$ strictly greater than the following values: $\frac{1}{\eps}$, $\den(t)$, and the absolute value of each coordinate of $v'$ and $v'+w'$.  It follows that $\frac{1}{p}<\eps$, the integers $p$ and $\den(t)$ are co-prime, and hence one easily obtains that $\den\mleft(t+\frac{v'}{p}\mright)=\den(t)p$. Furthermore, $\den(\frac{v'}{p})=\den(\frac{v'+w'}{p})=p$ (using the fact that $v',v'+w'\neq 0$).  Set $q\df \frac{1}{p}$. Then $q\in(0,\eps)$ and 
\[\den(t+q v')=\den\mleft(t+\frac{v'}{p}\mright)=\den(t)p=\den\mleft(t+\frac{v'+w'}{p}\mright)=\den(t+q v' +q w').\]
To conclude the proof, set $\delta \df \frac{q}{a}$ and $\tau \df \frac{q}{b}$.
\end{proof}

Informally, the next lemma asserts the following.
Let $S$ be a simplex in $\R^n$, $t \in S$ and $w_1,\dots w_l$ be vectors parallel to $S$.
Suppose that, for every ($k-1$)-dimensional face $F$ of $S$ to which $t$ belongs, for every $i$ the vector $w_i$ points towards $S$ when placed in the relative interior of $F$, and towards the relative interior of $S$ for at least one $i$.
Then moving from $t$ toward the direction given by a suitably small positive linear combination of $w_1,\dots, w_l$ gives a point in the relative interior of $S$.

We use the notation $\cdot$ for the scalar product of vectors. 

\begin{lemma} \label{l:lies-in-interior}
    Let $S \subseteq \R^n$ be a $k$-dimensional simplex and $t\in S$. Suppose that $S$ is described by the following system:
    \begin{align*}
        \begin{cases}
        \alpha_i\cdot x\geq a_i & \text{for }i\in\{0,\dots,k\},\\
        \alpha_i\cdot x= a_i & \text{for } i\in\{k+1,\dots,n\}.
        \end{cases}
    \end{align*}
If  $w_1, \dots, w_l$ are vectors parallel to $S$ such that for all $i \in \{0, \dots, k\}$ with $\alpha_i \cdot t = a_i$ and for all $j \in \{1, \dots, l\}$ we have 
\[\alpha_i \cdot w_j \geq 0\] 
and there is $j_0 \in \{1, \dots, l\}$ such that $\alpha_i \cdot w_{j_0} > 0$, then there is $\eps > 0$ such that, for all $\delta_1, \dots, \delta_l \in (0, \eps)$, $t + \delta_1 w_1 + \dots + \delta_l w_l$ is in the relative interior of $S$.
\end{lemma}

\begin{proof}
    For any $\delta_1, \dots, \delta_l >0$, $t + \delta_1 w_1 + \dots + \delta_l w_l$ is in the relative interior of $S$ if and only if
    \begin{align*}
        \begin{cases}
        \alpha_i\cdot (t + \delta_1 w_1 + \dots + \delta_l w_l)> a_i & \text{for }i\in\{0,\dots,k\},\\
        \alpha_i\cdot (t + \delta_1 w_1 + \dots + \delta_l w_l)= a_i & \text{for } i\in\{k+1,\dots,n\}.
        \end{cases}
    \end{align*}
    For any $i \in \{k +1, \dots, n\}$ the following conditions hold: $\alpha_i \cdot t = a_i$ (since $t \in S$) and, for every $j \in \{1, \dots, l\}$, $\alpha_i \cdot w_j = a_i$ (since $w_j$ is parallel to $S$); therefore, $\alpha_i\cdot (t + \delta_1 w_1 + \dots + \delta_l w_l)= a_i$.
    Therefore, $t + \delta_1 w_1 + \dots + \delta_l w_l$ is in the relative interior of $S$ if and only if, for every $i\in\{0,\dots,k\}$, $\alpha_i\cdot (t + \delta_1 w_1 + \dots + \delta_l w_l)> a_i$.
    Fix $i \in \{0, \dots, k\}$.
    If $\alpha_i \cdot t > a_i$, then for $\delta_1, \dots, \delta_l$ small enough we have $\alpha_i\cdot (t + \delta_1 w_1 + \dots + \delta_l w_l) = (\alpha_i\cdot t) + \delta_1 (\alpha_i \cdot w_1) + \dots + \delta_l (\alpha_i \cdot w_l) > a_i$.
    If, otherwise, $\alpha_i \cdot t = a_i$, then by hypothesis there is $j_0 \in \{1, \dots, l\}$ such that $\alpha_i \cdot w_{j_0} > 0$ and so,
    whatever positive values $\delta_1, \dots, \delta_l$ have, we have
    \[
        \alpha_i\cdot (t + \delta_1 w_1 + \dots + \delta_l w_l) = (\alpha_i \cdot t) + \delta_1 (\alpha_i \cdot w_1) + \dots + \delta_l (\alpha_i \cdot w_l) \geq a_i + \delta (\alpha_i \cdot w_{j_0}) > a_i. \qedhere
    \]
\end{proof}

The next lemma is crucial. We are going to prove that, if $S$ is a rational simplex of dimension $k\geq 2$ and $\eta\colon S \to \R$ is an affine map with integer coefficients that is not constant on some $(k-1)$-dimensional face $F$ of $P$, then $\eta$ can be \emph{squeezed} on a smaller domain without essentially changing its behaviour. 

The idea of the proof is to show the existence of a $\Z$-map that can slightly bend the face $F$ inside $S$.
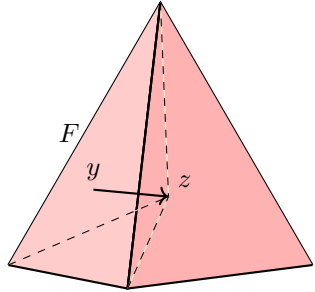
\begin{center}

\begin{tikzpicture}[scale=4]
\coordinate (A) at (0,0,0);
\coordinate (B) at (1,0,0);
\coordinate (C) at (0.5,0.87,0);
\coordinate (D) at (0.7,0.23,0.8);
\coordinate (E) at (0.2,0.5,0); 
\draw[thick] (A) -- (C) -- (B);
\draw[thick,fill=red!20] (A) -- (D) -- (C);
\draw[thick,fill=red!30] (B) -- (D) -- (C);
\draw[thick] (C) -- (D);
\coordinate (F) at (0.7,0.4,0.45);
\coordinate (G) at (0.53,0.5,0.65);
\draw[dashed,fill=red!20] (A) -- (F);
\draw[dashed,fill=green!20] (D) -- (F);
\draw[dashed,fill=yellow!20] (C) -- (F);
\draw[thick,->] (G) -- (F);
\node[above right] at (F) {$z$};
\node[above] at (G) {$y$};
\node at (G) {};
\node[below] at (E) {$F$};
\end{tikzpicture}
\end{center}
The key is to find a point $y$ in the relative interior of $F$ and a point $z$ in $S\setminus F$ with the same denominator and on which $\eta$ assumes the same value.
To do so, we start with an arbitrary point $t$ in the relative interior of $F$. This may not be the desired point $y$, as $S \setminus F$ may lack points with the same denominator.
However, there is a point $y$ (which in our proof will be $t + \delta v$ for an appropriate vector $v$ and a positive real number $\delta$) close to $t$ that does the job: it will have the same denominator of a point $z$ (which in our proof will be $t + \delta v + \tau w$ for an appropriate vector $w$ and a positive real number $\tau$) in the relative interior of $S$.

\begin{lemma}[Squeezing lemma]\label{l:retract}
Let $k\geq 2$, let $S$ be a rational $k$-simplex in $\R^n$, let $F$ be a $(k-1)$-dimensional face of $S$, and let $\eta\colon \R^n\to \R$ be an affine map with integer coefficients that is not constant on $F$. There is a non-surjective $\Z$-map $\rho\colon S\to S$ such that the diagram
\[
\begin{tikzcd}
{S}\arrow{r}{\eta}&\R\\
{S}\arrow[dashed]{u}{\rho}\arrow[swap]{ur}{\eta}&
\end{tikzcd}
\]
commutes and, for every $(k-1)$-dimensional face $G \neq F$ of $S$ and $x\in G$, $\rho(x)=x$.
\end{lemma}
\begin{proof}
By \cref{t:hull-intersection} the simplex $S$ can be described by a set of inequalities and equalities:
\[
\begin{cases}
\alpha_i\cdot x\geq a_i & \text{for }i\in\{0,\dots,k\},\\
\alpha_j\cdot x= a_j & \text{for } j\in\{k+1,\dots,n\}.
\end{cases}
\]
Since $S$ is a rational simplex, the coefficients are rational.
Without loss of generality, we can assume that $F$ is given by
\begin{align}\label{eq:faceF}
\begin{cases}
\alpha_0\cdot x= a_0,&\\
\alpha_i\cdot x\geq a_i &\text{for }i\in\{1,\dots,k\},\\
\alpha_j\cdot x= a_j &\text{for }j\in\{k+1,\dots,n\}.
\end{cases}
\end{align}
The relative interior of $F$ is then described by making the inequalities in \eqref{eq:faceF} strict:
\begin{equation}\label{eq:faceF-strict}
\begin{cases}
\alpha_0\cdot x= a_0,&\\
\alpha_i\cdot x> a_i &\text{for }i\in\{1,\dots,k\},\\
\alpha_j\cdot x= a_j &\text{for }j\in\{k+1,\dots,n\}.
\end{cases}
\end{equation}
Let $V_S$, $V_F$ and $V_\eta$ be the directions of $S$, $F$ and $\ker{\eta} \coloneqq \{x \in \R^n \mid \eta(x) = 0\}$, respectively. 
Then,
\begin{align*}
    V_S &= \{u \in \R^{n} \mid \alpha_{j} \cdot u = 0 \text{ for }j\in\{k+1,\dots,n\} \} \text{ and} \\
    V_F &= \{ u \in \R^n \mid \alpha_0 \cdot u = 0,\; \alpha_{j} \cdot u = 0 \text{ for }j\in\{k+1,\dots,n\}\}.
\end{align*}

Since $k \geq 2$, the dimension of the rational simplex $F$ is at least $1$ and so there is a rational point $t$ in the relative interior of $F$. In particular, \eqref{eq:faceF-strict} gives:
\[
\begin{cases}
\alpha_0\cdot t= a_0,&\\
\alpha_i\cdot t> a_i &\text{for }i\in\{1,\dots,k\},\\
\alpha_j\cdot t= a_j &\text{for }j\in\{k+1,\dots,n\}.
\end{cases}
\]
Let $v$ be a nonzero rational vector parallel to $F$, which exists because the dimension of $F$ is at least $1$ since $n \geq 2$. Since $v$ is parallel to $F$, we have
\[
\alpha_0 \cdot {v} = 0 \qquad\text{ and }\qquad    \alpha_{j} \cdot {v} = 0  \quad\text{ for all } j\in\{k+1,\dots,n\}.
\]

The next step in the proof is to find a point $y$ in the relative interior of $F$ and a point $z$ in $S\setminus F$ with the same denominator.

\begin{enumerate}[align=left, left=0pt, label = {\textbf{Claim 1.}}]
    \item
    There is a rational vector $w\in (V_S\cap V_\eta)\setminus V_F$. 
\end{enumerate} 

\noindent From $F \subseteq S$ we deduce $V_F \subseteq V_S$ and, since $\eta$ is not constant on $F$, $V_F \not\subseteq V_\eta$. Therefore, $V_S \not\subseteq V_\eta$. Thus, $\dim(V_S + V_\eta) > \dim(V_\eta) = n-1$ and thus $\dim(V_S + V_\eta) = n$. By Grassmann's identity, 
\[
    \dim(V_S \cap V_\eta) = \dim(V_S) + \dim(V_\eta) - \dim(V_S + V_\eta)= k + (n-1) -n = k -1.    
\]
Therefore, $V_F$ and $V_S \cap V_\eta$ have the same finite dimension and must be distinct because $V_F \not\subseteq V_\eta$. It follows that neither of the two is contained in the other one. Thus, $V_S \cap V_\eta \not\subseteq V_F$. This reasoning stays true if we work over the field of rational numbers: $(V_S \cap \Q) \cap (V_\eta \cap \Q) \not\subseteq (V_F \cap \Q)$, and the claim is proved.

Since $w \in V_S \setminus V_F$, $\alpha_0 \cdot {w} \neq 0$.
Up to possibly replacing ${w}$ with ${-w}$, we can suppose $\alpha_0 \cdot {w} > 0$.

With two applications of \cref{l:lies-in-interior}, we get: (i) there is $\eps' > 0$ such that, for all $\delta \in (0, \eps')$, the point $t + \delta v$ is in the relative interior of $F$, and (ii) there is $\eps'' > 0$ such that, for all $\delta, \tau \in (0, \eps'')$, $t + \delta v + \tau w$ is in the relative interior of $S$.
Therefore, we can fix $\eps > 0$ (take, for example, the minimum between $\eps'$ and $\eps''$) such that for all $\delta, \tau \in (0, \eps)$ the point $t + \delta v$ is in the relative interior of $F$ and $t + \delta v + \tau w$ is in the relative interior of $S$.

By \cref{l:hyperplane}, we can choose $\delta, \tau \in (0, \eps) \cap \Q$ such that, letting  $y \df t + \delta v$  and  $z \df t + \delta v + \tau w$, we have $\den(y)=\den(z)$. Then, $y$ is in the relative interior of $F$ and $z$ is in the relative interior of $S$.
Applying \cref{l:strong-triangulation} to $\{S,\{y\}\}$, we obtain a regular triangulation $\Delta$ of $S$ such that $y\in \vertices(\Delta)$. We define the function
\begin{align*}
    g \colon \vertices(\Delta) & \longrightarrow S\\
    x &\longmapsto {\begin{cases}
    z &\text{if }x=y,\\
    x&\text{otherwise.}
    \end{cases}}
\end{align*}
Since $\den(y)=\den(z)$, the function $g$ satisfies the hypothesis of \cref{l:unique-extension}. Therefore, it can be uniquely extended to a $\Z$-map $\rho\colon S\to S$ that is affine on each simplex of $\Delta$.
It is clear that $\rho$ is the identity on all ($k-1$)-dimensional faces of $S$ different from $F$. In addition, it is not surjective, because the point $y\in S$ is not in its image.
We claim that, for all $x \in \vertices(\Delta)$, $\eta(x) = \eta(g(x))$.
This is clear for $x \neq y$, since $g(x) = x$.
For the case $x = y$, it suffices to note that $w = \frac{1}{\tau}(z - y)$ and that $w$ is parallel to $\ker{\eta}$; therefore, $\eta(y) = \eta(z) = \eta(g(y))$.
This proves our claim that $\eta(x) = \eta(g(x))$.
Therefore, for all $x \in \vertices(\Delta)$, $\eta(x) = \eta(g(x)) = \eta(\rho(x))$.
Therefore, both $\eta$ and $\eta \circ \rho$ are $\Z$-maps from $S$ to $\R$ that coincide on $\vertices(\Delta)$ and are affine on each symplex of $\Delta$.
It follows that $\eta=\eta\circ \rho$. 
\end{proof}

\begin{lemma}\label{l:make-space}
Let $n\geq 2$, and let $\eta\colon [0,1]^n\to \R$ be a $\Z$-map that is not constant on the boundary of $[0,1]^n$. There is a non-surjective $\Z$-map $\alpha\colon [0,1]^n\to [0,1]^n$ that makes the following diagram commute.
\[
\begin{tikzcd}
{[0,1]^n}\arrow{r}{\eta}&\R\\
{[0,1]^n}\arrow[dashed]{u}{\alpha}\arrow[swap]{ur}{\eta}&
\end{tikzcd}
\]
\end{lemma}
\begin{proof}
By \cref{l:triangulation} there is a rational triangulation $\mathcal{K}$ of $[0,1]^n$ such that for every simplex $T\in\mathcal{K}$ there is an affine map $\eta_T$  with integer coefficients that coincides with $\eta$ over $T$.
Let $I$ be the image  under $\eta$ of the boundary of $[0,1]^{n}$. 
The set $I$ is not a singleton because $\eta$ is not constant on the boundary of $[0,1]^n$. Furthermore, since $n\geq 2$, the boundary of $[0,1]^n$ is connected, and hence $I$ is connected. It follows that $I$ contains infinitely many points.  Since $\KK$ contains finitely many $(n-1)$-simplices, there is one of them, say $F$, included in the boundary of $[0,1]^n$ and such that $\eta$ is not constant on it. 
Let $S$ be an $n$-simplex in $\KK$ that contains $F$. By \cref{l:retract}, there is a non-surjective $\Z$-map $\rho\colon S\to S$ such that
\begin{enumerate}[label=(\roman*), ref = \roman*]
\item\label{enum:property1} for every $x\in S$, $(\eta\circ\rho)(x)=\eta(x)$, and
\item\label{enum:property2} for every ($k-1$)-dimensional face $G$ of $S$ different from $F$ and every $x\in G$, $\rho(x)=x$.
\end{enumerate} 
We set
\begin{align*}
\alpha\colon [0,1]^n&\longrightarrow [0,1]^n\\
x&\longmapsto {\begin{cases}
\rho(x)&\text{if }x\in S,\\
x&\text{otherwise.}
\end{cases}}
\end{align*}
The continuity of $\alpha$ follows from \eqref{enum:property2}. Furthermore, $\alpha$ is a non-surjective $\Z$-map because $\rho$ is such a map. Finally, \eqref{enum:property1} implies that $\eta\circ\alpha=\eta$.
\end{proof}

We note that \cref{l:make-space} would fail for $n = 1$, as one can see by taking $\eta$ to be the inclusion of $[0,1]$ into $\R$.

The importance of \cref{l:make-space} lies in the fact that $\alpha$ is non-surjective; this guarantees that if we modify the value of $\eta$ only on points that are not in the image of $\alpha$ so to get a $\Z$-map $\theta$, then $\theta$ will be more general than $\eta$, because the following diagram will commute.
\[
\begin{tikzcd}
{[0,1]^n}\arrow{r}{\theta}&\R\\
{[0,1]^n}\arrow[dashed]{u}{\alpha}\arrow[swap]{ur}{\eta}&
\end{tikzcd}
\]
And if we succeed in letting $\theta$ gain a value that is not in the image of $\eta$, then $\theta$ will be \emph{strictly} more general than $\eta$.
This is the point of the following result.

\begin{theorem}\label{t:exists-upper-bound}
Let $n\geq 2$, let $\eta\colon [0,1]^n\to \R$ be a $\Z$-map that is not constant on the boundary of $[0,1]^n$, and let $z\in \Z$.
There are $\Z$-maps $\theta\colon [0,1]^n\to \R$ and $\alpha\colon [0,1]^n\to [0,1]^n$ such that $z$ belongs to the image of $\theta$ and the following diagram commutes.
\[
\begin{tikzcd}
{[0,1]^n}\arrow{r}{\theta}&\R\\
{[0,1]^n}\arrow[dashed]{u}{\alpha}\arrow[swap]{ur}{\eta}&
\end{tikzcd}
\]
\end{theorem}
\begin{proof}
By \cref{l:make-space}, there is a non-surjective $\Z$-map $\alpha\colon [0,1]^n\to [0,1]^n$ such that $\eta\circ\alpha=\eta$.
Let $P$ be the rational polyhedron given by the image of $\alpha$. Since $\alpha$ is not surjective, $P$ is strictly contained in $[0,1]^n$. By \cref{l:extend-super}, the restriction of $\eta$ to $P$ can be extended to a map $\theta\colon [0,1]^n\to \R$ that attains the value $z$. Finally, $\theta\circ\alpha=\eta\circ\alpha=\eta$.
\end{proof}

\section{Universal covering and the degree of a unifier}
\label{s:strictly-more-general}

Our next goal is to prove that the $\Z$-map $\theta$ obtained in \cref{t:exists-upper-bound} is strictly more general than $\eta$. To achieve this, we will build on the results of \cite{AbbaDiSpada}. Consider the map $\zeta\colon \R\to \B$ that wraps $\R$ around $\B$, counter-clockwise, at constant speed 1, sending $0$ to $(0,0)$:
\begin{equation}\label{eq:def-zeta}
\begin{split}
\zeta\colon \R&\longrightarrow \B\\
x&\longmapsto{\begin{cases}
(x-\lfloor x\rfloor,0)&\text{if }\lfloor x\rfloor\equiv 0 \mod 4,\\
(1,x-\lfloor x\rfloor)&\text{if }\lfloor x\rfloor\equiv 1 \mod 4,\\
(1-(x-\lfloor x\rfloor),1)&\text{if }\lfloor x\rfloor\equiv 2 \mod 4,\\
(0,1-(x-\lfloor x\rfloor))&\text{if }\lfloor x\rfloor\equiv 3 \mod 4,
\end{cases}}
\end{split}
\end{equation}
where $\lfloor x\rfloor$ is the greatest integer below $x$.
As proved in \cite[Lemma 3]{AbbaDiSpada}, the map $\zeta\colon \R\to \B$ is the \emph{universal cover} of $\B$ (see e.g., \cite[Section 1.3]{hatcher}).  In particular, this means that for every continuous map $\eta\colon [0,1]^n\to \B$ there is a \emph{lift} $\tilde{\eta}$ of $\eta$ to $\R$, i.e.\ a continuous function $\tilde{\eta} \colon [0,1]^n \to \R$ such that $\eta=\zeta\circ\tilde{\eta}$.

\begin{remark}\label{r:zeta-Z}
Notice that the map $\zeta\colon \R\to \B$ is continuous but is not a $\Z$-map; however, its restriction $\zeta_{a,b}$ to any closed interval $[a,b]$ with $a,b \in \Z$ is such a map. Indeed, $\zeta_{a,b}$ is defined on every interval $[a+i, a+i+1]$, with $0 \leq i < \lvert b-a \rvert$, by one of the four cases of \eqref{eq:def-zeta}; each of these functions is affine with integer coefficients on $[a+i, a+i+1]$ because $x-\floor{x}$ can be written on such an interval as $ x-(a+i)$.
\end{remark}
Every continuous map $\eta \colon [0,1]^n \to \B$ has infinitely many lifts $[0,1]^n \to \R$. However, they all have a connected image and all share the same length in $\R$ (see \cite[Lemma 5]{AbbaDiSpada}). We denote the length of the image of any lift of  $\eta$ by $\dg(\eta)$ (the \word{degree} of $\eta$).

\begin{lemma}[{\cite[Lemma 7]{AbbaDiSpada}}]\label{l:degree-factor}
If a $\Z$-map $\theta \colon  [0,1]^m \to \B$ is more general than a $\Z$-map $\eta\colon [0,1]^n \to \B$, then $\dg(\eta)\leq\dg(\theta)$.
\end{lemma}

\begin{lemma}\label{l:not-constant-is-preserved}
Let $n\in \N$ and $\eta\colon [0,1]^n\to \B$ be a $\Z$-map that is not constant on $\{0,1\}^n$. For every $n' \in \N$, every $\Z$-map $\theta\colon [0,1]^{n'}\to \B$ that is more general than $\eta$ is not constant on $\{0,1\}^{n'}$.
\end{lemma}
\begin{proof}
Let $\theta\colon [0,1]^{n'}\to \B$ be a $\Z$-map that is more general than $\eta$, i.e.\ such that there is a $\Z$-map $\alpha \colon [0,1]^n\to [0,1]^{n'}$ such that $\eta=\theta\circ \alpha$. Let $a,b\in \{0,1\}^n$ be such that $\eta(a)\neq \eta(b)$. Then $\theta(\alpha(a))=\eta(a)\neq \eta(b)=\theta(\alpha(b))$. Since $\alpha(a),\alpha(b)\in\{0,1\}^{n'}$, this shows that $\theta$ is not constant on $\{0,1\}^{n'}$.
\end{proof}

\begin{lemma}\label{p:make-space}
Let $n\geq 2$, and let $\eta\colon [0,1]^n\to \B$ be a $\Z$-map that is not constant on the boundary of $[0,1]^n$. There is a $\Z$-map $\theta\colon [0,1]^n\to \B$ that is strictly more general than $\eta$.
\end{lemma}
\begin{proof}
Since $\zeta\colon \R\to \B$ is the universal cover of $\B$, the map $\eta$ admits a lift $\tilde{\eta}\colon [0,1]^n\to \R$. The function $\tilde{\eta}$ is a $\Z$-map by \cite[Lemma 5.3]{MSuni}. The image of $\tilde{\eta}$ is a compact subspace of $\R$. Let $z$ be an integer not in $\tilde{\eta}\big[[0,1]^n\big]$. Then, by \cref{t:exists-upper-bound}, there are a $\Z$-map $\theta\colon [0,1]^n\to \R$ and a $\Z$-map $\alpha\colon [0,1]^n\to [0,1]^n$ such that $z$ belongs to the image of $\theta$ and the following diagram commutes.
\[
\begin{tikzcd}
{[0,1]^n}\arrow{r}{\theta}&\R\\
{[0,1]^n}\arrow[dashed]{u}{\alpha}\arrow[swap]{ur}{\tilde{\eta}}&
\end{tikzcd}
\]
Composing with $\zeta$, we obtain that the following diagram commutes.
\[
\begin{tikzcd}
{[0,1]^n}\arrow{r}{\zeta\circ\theta}&\B\\
{[0,1]^n}\arrow[dashed]{u}{\alpha}\arrow[swap]{ur}{\eta}&
\end{tikzcd}
\]
Thus, $\zeta\circ\theta$ is more general than $\eta$. Moreover, $\dg(\eta)< \dg(\zeta\circ\theta)$ because the latter function attains an additional integer point among its values; by \cref{l:degree-factor}, it follows that $\eta$ is not more general than $\zeta\circ\theta$.
\end{proof}

\begin{lemma}\label{l:technical}
Let $n \in \N$ and let $\eta\colon [0,1]^n\to \B$ be a $\Z$-map that is not constant on $\{0,1\}^n$. 
For every $n' \in \N$ and every $\Z$-map $\theta\colon [0,1]^{n'}\to \B$ more general than $\eta$ there is a $\Z$-map $\psi \colon [0,1]^{\max\{2, n'\}}\to \B$ that is strictly more general than $\theta$.
\end{lemma}
\begin{proof}
    The function 
    \begin{align*}
    \theta' \colon [0,1]^{\max\{2, n'\}}& \longrightarrow \B\\
    (x_1, \dots, x_{\max\{2, n'\}}) & \longmapsto \theta(x_1, \dots, x_{n'})
    \end{align*}
    is more general than $\theta$, and so more general than $\eta$, too.
    By \cref{l:not-constant-is-preserved}, $\theta'$ is not constant on $\{0,1\}^{\max\{2, n'\}}$. Then we apply \cref{p:make-space} to get the conclusion.
\end{proof}

\begin{mtheorem}[Main Result]
    For all $m \geq 2$ and $n\geq 2$, the unification type of {\L}ukasiewicz logic restricted to at most $m$ variables for the problem and at most $n$ variables for the solutions is nullary.
\end{mtheorem}
\begin{proof}
Recall from \Cref{s:introduction} the unification problem $x_1 \lor x_2 \lor \lnot x_1 \lor \lnot x_2 \approx 1$ in variables $x_1$ and $x_2$ and its unifier $\iota$ in the variable $y$ defined by $\iota(x_1) \df y$ and $\iota(x_2) \df 0$. 
Under the duality of \cref{d:duality}, the unifier $\iota$ corresponds to 
\begin{align*}
\iota'\colon [0,1]&\hooklongrightarrow \B\seq[0,1]^2\\
x&\longmapsto (x,0).
\end{align*}
The function $\iota'$ is non-constant on $\{0,1\}$, because $\iota'(0) = 0 \neq 1 = \iota'(1)$.
By \cref{l:technical}, and since $n \geq 2$, among all $\Z$-maps  $[0,1]^k \to \B$ with $k \leq n$ there is no maximally general one that is more general than $\iota'$.
\end{proof}

\section*{Acknowledgments}
The first author's research was funded by UK Research and Innovation (UKRI) under the UK government’s Horizon Europe funding guarantee (grant number EP/Y015029/1, Project ``DCPOS''). The ``Horizon Europe guarantee'' scheme provides funding to researchers and innovators who were unable to receive their Horizon Europe funding (in this case, for a Marie Skłodowska-Curie Actions (MSCA) grant) while the UK was in the process of associating.

The second author acknowledges financial support under the National Recovery and Resilience Plan (NRRP), Mission 4, Component 2, Investment 1.1, Call for tender No. 1409 published on 14.9.2022 by the Italian Ministry of University and Research (MUR), funded by the European Union – NextGenerationEU– Project Title Quantum Models for Logic, Computation and Natural Processes (Qm4Np) – CUP F53D23011170001 - Grant Assignment Decree No. 1016 adopted on 07/07/2023 by the Italian Ministry of Ministry of University and Research (MUR).

The research of both authors was supported by the Italian Ministry of University and Research through the PRIN project n.~20173WKCM5 \emph{Theory and applications of resource sensitive logics}.

\bibliographystyle{abbrv}

\end{document}